%% file: indeq.tex
\newif\ifpub

\ifpub
\documentclass[twoside, 12pt]{article}
\else
\documentclass{amsart}
\fi

\ifpub
\usepackage{amssymb, amsmath, mathrsfs, amsthm}
\usepackage{newtxtext,newtxmath}
\usepackage{graphicx}
\usepackage{color}
\usepackage[top=1in, bottom=1in, left=1in, right=1in]{geometry}
\usepackage{float, caption, subcaption}
\newtheorem{theorem}{Theorem}[section]
\newtheorem{lemma}[theorem]{Lemma}
\newtheorem{definition}{Definition}[section]
\newtheorem{corollary}[theorem]{Corollary}

\newtheorem{proposition}{Proposition}[section]

\newtheorem{conjecture}{Conjecture}[section]
\else
\newtheorem{theorem}{Theorem}[section]
\newtheorem{proposition}[theorem]{Proposition}
\newtheorem{conjecture}[theorem]{Conjecture}
\newtheorem{definition}[theorem]{Definition}

\fi

\usepackage{stackengine}
\usepackage{tikz}
\usetikzlibrary{graphs}
\usetikzlibrary{graphs.standard}
\setstackgap{L}{.5\baselineskip}

\newcommand{\indeq}{\Vectorstack{\alpha{} \omega}}
\newcommand{\planar}{\mathcal{P}}

\tikzset{inner sep=0mm, minimum size=2mm,every node/.style=draw,circle}
\tikzset{invisible/.style={minimum size=0mm,inner sep=0mm,outer sep=0mm}}

\makeatletter
\g@addto@macro\@floatboxreset{\centering}
\makeatother

\ifpub

\begin{document}
\title{\bf Packing independent cliques into planar graphs}
\author{
Csaba Bir\'o\footnote{Department of Mathematics, University of Louisville, Louisville, KY 40292, USA; and Centro de Investigaci\'on en Matemática Pura y Aplicada, Universidad de Costa Rica, San Jos\'e, Costa Rica {\tt csaba.biro@louisville.edu}},
Gabriel Collado\footnote{Universidad de Costa Rica, San Jos\'e, Costa Rica; current address: Department of Mathematics, Statistics, and Computer Science, University of Illinois Chicago, Chicago, IL 60607, USA {\tt gcoll8@uic.edu}},
Oscar Zamora\footnote{Centro de Investigaci\'on en Matem\'atica Pura y Aplicada, Escuela de Matem\'atica, Universidad de Costa Rica, San Jos\'e, Costa Rica {\tt oscar.zamoraluna@ucr.ac.cr}}
}

\else

\title{Packing independent cliques into planar graphs}

\author{Csaba Bir\'o}
\address{Department of Mathematics, University of Louisville, Louisville, KY 40292, USA}
\address{Centro de Investigaci\'on en Matemática Pura y Aplicada, Universidad de Costa Rica, San Jos\'e, Costa Rica}
\author{Gabriel Collado}
\address{Universidad de Costa Rica, San Jos\'e, Costa Rica}
\curraddr{Department of Mathematics, Statistics, and Computer Science, University of Illinois Chicago, Chicago, IL 60607, USA}
\author{Oscar Zamora}
\address{Centro de Investigaci\'on en Matem\'atica Pura y Aplicada, Escuela de Matem\'atica, Universidad de Costa Rica, San Jos\'e, Costa Rica}

\begin{document}
\fi

\begin{abstract}
The indeque number of a graph is the size of a largest set of vertices that induces an independent set of cliques. We study the extremal value of this parameter for the class and subclasses of planar graphs, most notably for forests and graphs of pathwidth at most $2$.
\end{abstract}

\ifpub\thispagestyle{empty}\else\maketitle\fi

\input{intro}
\input{forests}

\input{pathwidth}

\input{planar}

\input{further}

\ifpub
\bibliographystyle{abbrv}
\else
\bibliographystyle{plain}
\fi

\bibliography{indeq}

\end{document}

%% file: intro.tex
\section{Introduction}

In this article we study the following graph parameter.

\begin{definition}
Let $G$ be a graph. A set $S\subseteq V(G)$ is an \emph{indeque set}, if $G[S]$ is disjoint union of independent cliques. The \emph{indeque number} of $G$ is
\[
\indeq(G)=\max\{|S|:S\text{ is an indeque set of $G$}\}.
\]
\end{definition}

The name is motivated by the trivial inequalities
\[
\max\{\alpha(G),\omega(G)\}\leq \indeq(G)\leq\alpha(G)\omega(G),
\]
which express that the indeque number is no less than the independent number and the clique number, but no greater than their product. The gap in both inequalities can be arbitrarily large: indeed, for $n\in\mathbb{N}$, let $G_n$ be graph on $3n$ vertices with $n$ independent edges. Then $\alpha(G_n)=2n$, $\omega(G_n)=2$, and $\indeq(G_n)=3n$.

This graph parameter in this form was first introduced by Ertem et al.~\cite{Ertem} (though not under this name), but it has been studied in disguise as the cluster vertex deletion problem. In that setting, one is looking for the minimum number of vertices to delete from a graph such that the resulting graph is a collection of independent cliques.

This problem is algorithmically difficult, even in some very limited settings. To find out if there is an indeque set of size at least $k$ (where $k$ is part of the input) is NP-complete \cite{NPC}, even when restricted to bipartite graphs \cite{bip}, or even bipartite graphs of maximum degree $3$ \cite{bip3}. More relevant to us that the problem remains NP-complete for planar graphs \cite{planar}. Very recently, Le and Le \cite{LeLe} characterized the graphs for which the cluster vertex deletion problem is polynomially solvable in terms of forbidden induced subgraphs.

\begin{definition}
Let $\mathcal{G}$ be a class of graphs. Define
\[
\indeq(\mathcal{G},n)=\min\{\indeq(G): G\in\mathcal{G}\text{ and }|G|=n\}
\]
and
\[
\indeq(\mathcal{G})=\lim_{n\to\infty}\frac{\indeq(\mathcal{G},n)}{n}
\]
if the limit exists. The number $\indeq(\mathcal{G})$ is called the \emph{indeque ratio} of the class $\mathcal{G}$.
\end{definition}

If the graph class $\mathcal{G}$ is closed under disjoint union, then $\indeq(\mathcal{G},n+m)\leq\indeq(\mathcal{G},n)+\indeq(\mathcal{G},m)$, so by Fekete's Lemma,
$\lim\indeq(\mathcal{G},n)/n$ exists, and equals to $\inf\{\indeq(\mathcal{G},n)/n\}$. In this paper, all our studied classes are closed under disjoint union, so the indeque ratio in all cases is well-defined.

Salia et al.~\cite{salia} studied the indeque ratio of the family of comparability graphs of posets that have acyclic cover graph. Motivated by the open problem posed at the end of the paper, we studied the indeque ratio of planar graphs, and subclasses of planar graphs.

Let $\planar$ be the class of planar graphs. By the Four Color Theorem, $\indeq(\planar,n)\geq n/4$, so $\indeq(\planar)\geq 1/4$. We will show that this lower bound is not sharp, though the exact value of $\indeq(\planar)$ remains open. Also note that any lower bound for $\indeq(\planar)$ serves as a lower bound for any subclass of $\planar$. However, for some subclasses we were able to find the indeque ratio exactly.

In this article we will freely use the definitions and notations of Diestel's \cite{D} excellent graph theory textbook. In particular the following definitions will be essential in Section~\ref{sec:pathwidth}.

\begin{definition}
Let $G$ be a graph. A \emph{tree-decomposition} of $G$ is a tree $T$ whose vertices are subsets of $V(G)$ satisfying the following properties.
\begin{itemize}
\item For every vertex $v\in V(G)$ there is a $B\in V(T)$ such that $v\in B$.
\item For every edge $\{u,v\}\in E(G)$ there is a $B\in V(T)$ such that $u,v\in B$.
\item For every vertex $v\in V(G)$, the collection $\{B\in V(T):v\in B\}$ induces a (connected) subtree of $T$.
\end{itemize}
The \emph{width} of a tree-decomposition is 
\[
\max_{B\in V(T)}\{|B|-1\}.
\]
\end{definition}
If $T$ is a path, the tree-decomposition is called a \emph{path-decomposition}.
\begin{definition}
The \emph{treewidth} of a graph $G$ is the minimum width of a tree-decomposition of $G$. The \emph{pathwidth} of a graph $G$ is the minimum width of a path-decomposition of $G$.
\end{definition}

%% file: forests.tex
\section{Forests}

In this section, we study the indeque ratio of the class of forests, which we will denote by $\mathcal{F}$.

\begin{theorem}\label{thm:forest}
\[\indeq(\mathcal{F})=2/3\]
\end{theorem}
\begin{proof}
It is easy to see that for the path $P_{n-1}$ (that is, a path on $n-1$ edges and $n$ vertices), if $3|n$, then $\indeq(P_{n-1})=2n/3$. This shows $\indeq(\mathcal{F})\leq 2/3$.

For the lower bound, we will proceed by induction. For small forests, the statement of the theorem is easily verified. Assume $F$ is a forest, and let $T$ be a largest component. Since indeque sets can be constructed componentwise, we may assume $|T|\geq 4$.

Let $P$ be a longest path in $T$, and let $v$ be an endpoint of $P$. Note that $v$ is a leaf of $T$, and so it has a unique neighbor $u$. We distinguish two cases.

If $\deg(u)=2$, then let $w$ be the unique vertex such that $w\neq v$ and $w\sim u$. Let $T'=T-\{u,v,w\}$. By the hypothesis, $T'$ has an indeque set $S$ such that $|S|\geq 2|T'|/3$. Since $S\cup\{u,v\}$ is an indeque set in $T$, we get that $\indeq(T)\geq |S|+2\geq 2|T'|/3+2=2|T|/3$.

If $\deg(u)>2$, then let $w$ be a neighbor of $u$ that is not on $P$. Note that the maximality of the length of $P$ implies that $w$ is a leaf. Let $T'=T-\{u,v,w\}$ and $S$ be an indeque set of $T'$. Since $S\cup\{v,w\}$ is an indeque set of $T$, the rest of the proof is the same as in the previous case.
\end{proof}

%% file: pathwidth.tex
\section{Graphs of small pathwidth}\label{sec:pathwidth}

Let $\mathcal{P}_k$ be the class of graphs of pathwidth at most $k$. Since $\mathcal{P}_1$ contains only forests, but the paths are contained in it, our previous results imply that $\indeq(\mathcal{P}_1)=2/3$. The class $\mathcal{P}_3$ contains non-planar graphs (e.g.\ $K_{3,3}$), so our interest is to determine $\indeq(\mathcal{P}_2)$.

We prove the following theorem.

\begin{theorem}\label{thm:pw}
\[
\indeq(\mathcal{P}_2)=1/2
\]
\end{theorem}

To prove our results, we will use two structural theorems. The first one is a well-known theorem about the connectivity structure of general graphs.

Recall that a \emph{block} is a maximal connected subgraph without a cutvertex, i.e.~a maximal $2$-connected subgraph or a bridge (with its ends) or an isolated vertex. Every graph can be decomposed into its blocks, of which any two share at most one vertex, and each shared vertex is a cutvertex of $G$. Let $A$ be the set of cutvertices of $G$, and $\mathcal{B}$ be the set of its blocks. The \emph{block graph} is a bipartite graph on $A\cup\mathcal{B}$, whose edge set is $\{\{a,B\}:a\in A, B\in\mathcal{B}, a\in B\}$. The following theorem is proven in \cite{D}. (See Lemma~3.1.4.)

\begin{theorem}
The block graph of any graph is a forest.
\end{theorem}

The second theorem that we will need is about the structure of maximal 2-connected subgraphs of graphs of pathwidth at most $2$. This was proven by Bir\'o, Keller, and Young \cite{BKY}. Previously, Bar\'at et al.~\cite{Betal} proved a similar, but somewhat less general theorem.

\begin{theorem}\label{thm:pwstruct}
Let $G$ be a graph of pathwidth at most $2$, and let $B$ be a 2-connected block of $G$. Then $B$ has the following structure.
\begin{itemize}
\item $B$ has a longest cycle $C$ on vertices (in order) $v_1,v_2,\ldots,v_k,w_\ell,w_{\ell-1},\ldots,w_1$ with $k,\ell\geq 1$, and $k+\ell\geq 3$.
\item $B$ consists of $C$ and some $C$-paths (non-trivial paths that meet $C$ exactly in its ends) of length at most $2$ (i.e.~having one or two edges).
\item All $C$-paths are from some $v_i$ to some $w_j$, such that for two $C$-paths $v_{i_1}P_1w_{j_1}$ and $v_{i_2}P_2 w_{j_2}$ for which $i_1<i_2$, we have $j_1\leq j_2$.
\item If there is a vertex $b\in V(B)$ that is adjacent to a vertex $v\not\in V(B)$, then $b\in V(C)$.
\end{itemize}
\end{theorem}

Let $B$ be a $2$-connected block of a graph in $\mathcal{P}_2$ with the structure described above identified and fixed. Let $i$ be the least integer for which there is a $C$-path $P$ from $v_i$ to $w_j$, if such exists. The cycle $v_iPw_jw_{j-1}\ldots w_1v_1\ldots v_{i}$ will be called the \emph{left end-cap} of $B$. Similarly, if $i$ is the greatest integer for which there is a $C$-path $P$ from $v_i$ to $w_j$, the cycle $v_iPw_jw_{j+1}\ldots w_\ell v_k\ldots v_{i}$ will be called the \emph{right end-cap} of $B$. (Figure~\ref{fig:endcaps}.) Note that unless $B$ is a single cycle, both end-caps exist and are distinct, however their vertex sets are not necessarily disjoint.

\begin{figure}
\input{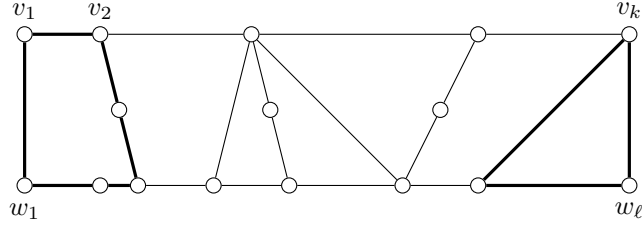}
\caption{A $2$-connected block of a graph of pathwidth at most $2$. The two end-caps are marked by thick lines.}
\label{fig:endcaps}
\end{figure}

We are ready to prove Theorem~\ref{thm:pw}.

\begin{proof}
Since $\indeq(C_4)=2$ and $C_4\in\mathcal{P}_2$,
it follows that $\indeq(\mathcal{P}_2)\leq 1/2$.

For the lower bound we proceed by induction on $|G|$ with the small cases being trivial. Now suppose $G\in\mathcal{P}_2$, and the theorem has been established for graphs on fewer vertices than $|G|$.

If $G$ has an isolated vertex $v$, then $v$ can be added to an indeque set of $G-v$, so $\indeq(G)\geq \indeq(G-v)+1$, and the hypothesis finishes the proof.

If $G$ has a leaf $v$, and $w\sim v$, then apply the hypothesis for $G-v-w$. If $S$ is an indeque set of $G-v-w$, then $S\cup\{v\}$ is an indeque set of $G$, and so $\indeq(G)\geq\indeq(G-v-w)+1$, finishing the proof.

Thus, we can assume that $\delta(G)\geq 2$. Let $B$ be a leaf or an isolated vertex in the block graph of $G$. We have shown that $B$ (as a subgraph of $G$) is not a bridge or an isolated vertex, so it is a $2$-connected subgraph of $G$. Also, $B\in\mathcal{P}_2$, so Theorem~\ref{thm:pwstruct} applies, but furthermore, $B$ has at most one vertex that is adjacent to a vertex outside of $B$ in $G$; call this vertex (if exists) $b$. According to Theorem~\ref{thm:pwstruct}, $b$ is on a longest cycle $C$ of $B$. In the rest of the proof, we will assume that $b$ exists, otherwise a trivial modification will make the proof simpler.

Suppose $B$ is a single cycle, and let $V(B)=\{b,v_1,\ldots,v_k\}$, in this order around the cycle. Let $S$ be an indeque set of $G-B$. Then $S\cup\{v_1,v_2,v_4,v_5,v_7,v_8\ldots\}$ is an indeque set of $G$, and has size at least $|S|+(k+1)/2$. So applying the hypothesis for $G-B$ finishes the proof. (Figure~\ref{fig:singlecycle}.)

\begin{figure}
\input{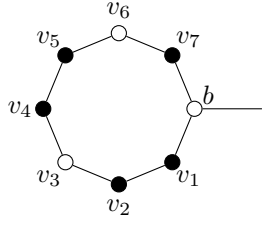}
\caption{$B$ is a single cycle. Additions to the indeque set are marked by filled dots.}
\label{fig:singlecycle}
\end{figure}

Now assume that $B$ is not a single cycle, so it has end-caps. Each end-cap consists of two paths between the vertices $v_i$ and $w_j$ we identified in Theorem~\ref{thm:pwstruct}. Refer to the path that runs along $C$ as $Q$, and the other one as $P$, as we did in the theorem. So $P$ is a $v_i$--$w_j$ path, and $Q$ is a $w_j$--$v_i$ path.

Note that we may assume that $b$ is not an internal vertex of $Q$. This is because $B$ has two distinct end-caps, and the internal vertices of the corresponding $w_j$--$v_i$ paths are disjoint, so $b$ can be in at most one of them. Note, however that we can not rule out that $b=v_i$ or $b=w_j$. 

Suppose that the internal vertices of $Q$ are (in order) $u_1,\ldots,u_m$, and $m\geq 2$. If $m=2$, then let $u_3=v_i$. Let $S$ be an indeque set of $G-\{w_j,u_1,u_2,u_3\}$. Then $S\cup\{u_1,u_2\}$ is an indeque set of $G$, so the proof is complete.

The remaining case is that $Q$ has at most one internal vertex. We know from Theorem~\ref{thm:pwstruct} that $P$ has at most one internal vertex, which is not $b$. At least one of $P$ and $Q$ has an internal vertex, and if both have one (say, $u_1$, $u_2$), then applying the hypothesis for $G-P-Q$ and adding $u_1$, $u_2$ to the resulting indeque set finishes the proof.

So it remains to resolve the case when $P$ and $Q$ have exactly one internal vertex together. This vertex $u$ is on $Q$, otherwise $C$ is not a longest cycle. Recall the structure of $B$: the path along the cycle $C$ from $w_j$ to $v_i$ is $Q$, but there is another path $Q'$ along $C$ from $v_i$ to $w_j$, which is internally disjoint from $Q$, and it has at least one interval vertex. This shows that $\deg(v_i)\geq 3$ and $\deg(w_j)\geq 3$. We claim that $\deg_B(v_i)=3$ or $\deg_B(w_j)=3$. Otherwise, there are indices $i',j'$ such that there is a $v_i$--$w_{j'}$ $C$-path and a $w_j$--$v_{i'}$ $C$-path in $B$. These paths contradict the third condition of Theorem~\ref{thm:pwstruct}. (Figure~\ref{fig:deg3}.)

\begin{figure}
\input{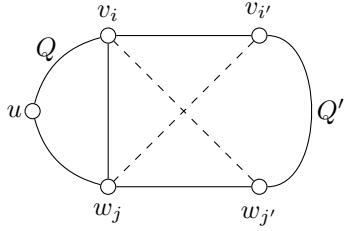}
\caption{The two dashed paths can't exist together}
\label{fig:deg3}
\end{figure}

Without loss of generality, let $\deg_B(v_i)=3$. Let the first internal vertex on the path $Q'$ be $z$. Apply the hypothesis for $G-\{v_i,w_j,u,z\}$; if $S$ is an indeque set of this graph, then $S\cup\{v_i,u\}$ is an indeque set of $G$, unless $b=v_i$. So this finishes the proof, unless $b$ is the degree-$3$ (in $B$) endpoint of $P$. (Figure~\ref{fig:Qsingle}.)

\begin{figure}
\input{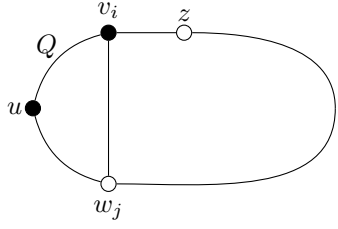}
\caption{Single internal vertex on $Q$}
\label{fig:Qsingle}
\end{figure}

In this situation, we can consider the other end-cap of $B$. Since $b$ can not be an internal vertex of that end-cap, the same argument can be used. We can finish the proof, unless we end up in the same subcase. A moment of consideration shows that in this case $|B|=4$, and $B$ is as depicted in Figure~\ref{fig:B4}.

\begin{figure}
\input{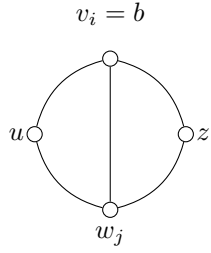}
\caption{$b$ is on both end-caps}
\label{fig:B4}
\end{figure}

This time, apply the hypothesis for $G-B$, and add $\{u,w_j\}$ to the indeque set. Note that this is the same as swapping the role of $v_i$ and $w_j$ in the previous argument.
\end{proof}

%% file: planar.tex
\section{The class $\planar$}

In this section, we discuss the indeque ratio of the class of planar graphs.

\subsection{Smallest known indeque number of a planar graph}

The following example (Figure~\ref{fig:oct}) is due to a person of unknown identity. The reason for this is that one of the authors of \cite{salia} was giving a talk on that paper, after which a member of the audience showed him the example. The audience member stated that he found the graph using a computer search.

\begin{figure}
\input{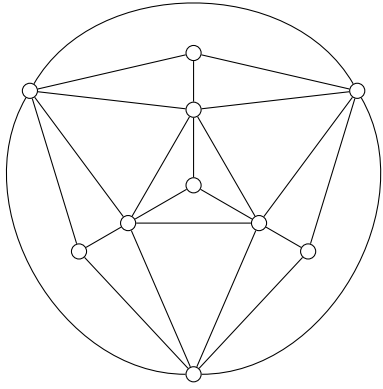}
\caption{Alternatingly apexiated octahedron}\label{fig:oct}
\end{figure}

Note that the graph in Figure~\ref{fig:oct} can be constructed from the octahedron by adding an apex vertex to every other face (an apex vertex is adjacent to every vertex of the boundary). Here ``every other face'' means that a face receives an apex vertex if and only if its adjacent faces do not. This is possible, because the dual of the octahedron is a bipartite graph.

The graph has several $4$-element indeque sets; in fact for every integer partition of $4$ it has a corresponding indeque set consisting of cliques of sizes of the parts. However, a simple case analysis (or computer search) shows that it has no $5$-element indeque set. So we proved the following proposition.

\begin{proposition}\label{prop:octa}
\[
\indeq(\planar)\leq 2/5
\]
\end{proposition}

For the other direction, we will use a classical theorem by Borodin. Recall that an acyclic coloring of a graph is a proper coloring such that every two color classes induce a forest. Borodin \cite{borodin} proved that every planar graph admits an acyclic $5$-coloring. We will use this result, and our earlier theorem to slightly improve on the upper bound of the indeque ratio of planar graphs.

\begin{proposition}\label{prop:planarlower}
\[
\indeq(\planar)\geq \frac{4}{15} 
\]
\end{proposition}

\begin{proof}
Let $G$ be a planar graph on $n$ vertices. Consider an acyclic $5$-coloring of $G$. Let $F$ be the union of the two largest color classes. Then $F$ is an induced forest of size at least $\frac25 n$.

By Theorem~\ref{thm:forest}, $F$ has an indeque set $S$ of size $\frac23 |F|\geq \frac4{15}n$. Since $S$ is also an indeque set of $G$, we have shown $\indeq(G)\geq\frac4{15}n$.
\end{proof}

We note that Albertson and Berman \cite{ab} conjectured that every planar graph has an induced forest containing at least half of the vertices. If true, this conjecture would imply a lower bound of $1/3$ for $\indeq(\mathcal{P})$. We will make a stronger conjecture in Section~\ref{sec:further}.

\subsection{Bounded girth}

Let $\mathcal{G}_k$ be the class of planar graphs of girth at least $k$. In this subsection, we, again, apply Theorem~\ref{thm:forest} to give lower bounds for the indeque ratio for classes $\mathcal{G}_4$ and $\mathcal{G}_5$.

Dross et al.~\cite{dross} proved that triangle free planar graphs (class $\mathcal{G}_4$) of order $n$ contain an induced forest of size $\frac{6n+7}{11}$. Using this, and our Theorem~\ref{thm:forest}, we conclude the following.

\begin{proposition}
\[
\frac{4}{11}\leq\indeq(\mathcal{G}_4)\leq \frac12
\]
\end{proposition}

\begin{proof}
Let $G$ be a planar graph on $n$ vertices of girth at least $4$. Then $G$ has an induced forest $F$ of size $\frac{6n+7}{11}$. By Theorem~\ref{thm:forest}, $F$ has an indeque set $S$ of size $\frac23|F|=\frac{12n+14}{33}$. Since $S$ is also an indeque set of $G$, we have shown $\indeq(G)\geq \frac{12n+14}{33}$, and hence $\indeq(\mathcal{G}_4,n)\geq \frac{12n+14}{33}$. Letting $n\to\infty$, we conclude $\indeq(\mathcal{G}_4)\geq\frac{12}{33}=\frac{4}{11}.$

The upper bound is demonstrated by the graph $C_4$.
\end{proof}

We do think the correct answer is closer to the upper bound. Perhaps it is equal to it.

For graphs of girth at least $5$, the existence of even larger sizes of induced forests were shown. Kelly and Liu \cite{kellyliu}, and independently, Shi and Xu \cite{shixu} proved that graphs of $\mathcal{G}_5$ admit an induced forest containing at least $2/3$ of the vertices. This gives the following bound.

\begin{proposition}\label{prop:G5}
\[
\indeq(\mathcal{G}_5)\geq\frac49
\]
\end{proposition}

Note that a promising result of Borodin and Glebov \cite{if} shows that every planar graph $G$ of girth at least $5$ has a decomposition of $V(G)$ into an independent set $I$ and an induced forest $F$. Since $I$ is an indeque set of $G$, and $F$ contains an indeque set of size $\frac23|F|$, we have that $m=\max\{|I|,\frac23|F|\}$ is a lower bound for $\indeq(G)$.

The number $m$ is least when $|I|=\frac23|F|$. A quick calculation shows that this happens when $m=\frac25|G|$, so the lower bound given by this argument for the indeque ratio of $\mathcal{G}_5$ is $2/5$. Unfortunately it is weaker than the one provided by Proposition~\ref{prop:G5}.

Kowalik et al.~\cite{kowalik} conjectured that elements of $\mathcal{G}_5$ of order $n$ have an induced forest of size $(7/10)n$. This would improve our lower bound to $7/15$.

%% file: further.tex
\section{Further questions}\label{sec:further}

Of course, the big open problem is to determine the indeque ratio of the class of planar graphs. We propose the following conjecture.

\begin{conjecture}
\[\indeq(\mathcal{P})=2/5\]
\end{conjecture}

In other words, we believe that Proposition~\ref{prop:octa} is sharp. To support this conjecture, we present another one.

Let $T_n=(V_n,E_n) $ be the graph that we will call triangular grid, where $T_n=\{(i,j):i+j\leq n\}$, and $(i,j)\sim(i',j')$ if $j=j'$ and $|i-i'|=1$; or $i=i'$ and $|j-j'|=1$; or $i-i'=j'-j=\pm 1$. So $T_0$ is a single vertex, $T_1$ is $K_3$, and $T_2$ is a graph on $6$ vertices and $9$ edges.

\begin{figure}
\input{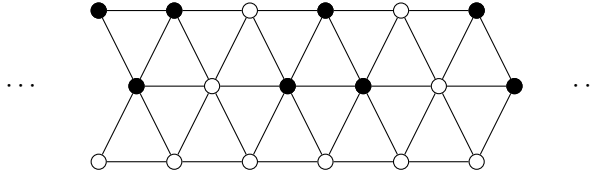}
\caption{Indeque set in $T_n$. The three rows of vertices are repeated.}\label{fig:triangles}
\end{figure}

\begin{proposition}
For all $n\in\mathbb{N}$,
\[
\indeq(T_n)\geq \frac25|T_n|,
\]
\end{proposition}

\begin{proof}
It is a relatively easy exercise to construct an indeque set of the proper size (see e.g.\ Figure~\ref{fig:triangles}). The pattern in Figure~\ref{fig:triangles} can be repeated to cover the infinite triangular grid, and therefore any finite subgraph of it. As $T_n$ is a subgraph of the infinite triangular grid, the pattern provides a construction of an indeque set of $T_n$. The finiteness of the grid sometimes makes it possible to find slightly larger indeque sets.
\end{proof}

Note that there are other patterns that provide indeque sets of similar size. Interestingly, we found a pattern that does not even use $K_3$ cliques.

We believe this is best possible.

\begin{conjecture}
\[
\lim_{n\to\infty}\frac{\indeq(T_n)}{|T_n|}=\frac25
\]
\end{conjecture}

\subsection{Questions on other subclasses}

Recall that graphs of treewidth at most $2$ are exactly the $K_4$-minor-free graphs, also known as series--parallel graphs. The techniques in Section~\ref{sec:pathwidth} can probably be generalized to this class, though the details may be formidable. Another option to find the indeque ratio is to use the recursive series--parallel construction, or perhaps the tree-decomposition directly. Our Theorem~\ref{thm:pw} implies that the indeque ratio is at most $1/2$. We are willing to risk a conjecture that it is the right answer.

\begin{conjecture}\label{conj:tw}
Let $\mathcal{S}$ be the class of series--parallel graphs. Then
\[
\indeq(\mathcal{S})=1/2.
\]
\end{conjecture}

Another related class is the class of outerplanar graphs. They can have large pathwidth, but their treewidth is at most $2$. Hence Conjecture~\ref{conj:tw} implies the following weaker conjecture.

\begin{conjecture}
Let $\mathcal{O}$ be the class of outerplanar graphs. Then
\[
\indeq(\mathcal{O})=1/2.
\]
\end{conjecture}

For this latter conjecture, we point out that Hosono \cite{hosono} proved that outerplanar graphs have an induced forest on at least $2/3$ on their vertices, which, with Theorem~\ref{thm:forest} imply $\indeq(\mathcal{O})\geq 4/9$. Hosono's theorem is best possible, but the extremal example does have an indeque set on half of the vertices, just not via an induced forest.